\documentclass{amsart}
\usepackage{amsmath,amssymb,amsfonts,amsthm}
\usepackage{amssymb}
\usepackage{pgf}
\usepackage{todonotes}
\usepackage{epsfig}
\usepackage{mathrsfs}
\usepackage{verbatim}
\newtheorem{theorem}{Theorem}[section]

\newtheorem{proposition}[theorem]{Proposition}

\theoremstyle{definition}
\newtheorem{definition}[theorem]{Definition}

\theoremstyle{remark}

\numberwithin{equation}{section}

\newcommand{\diam}{\mathrm{diam}}

\newcommand{\R}{\mathbb{R}}
\newcommand{\N}{\mathbb{N}}

\newcommand{\X}{\mathrm{X}}

\newcommand{\Y}{\mathrm{Y}}

\newcommand{\B}{\mathbf{B}}

\renewcommand{\S}{\mathbf{S}}

\newcommand{\vep}{\varepsilon}

\begin{document}

\title[Note on Bishop-Phelps-Bollob\'as property for operators]{Note on a kind of Bishop-Phelps-Bollob\'as property for operators}

\author{Jarno Talponen}
\address{University of Eastern Finland, Department of Physics and Mathematics, Box 111, FI-80101 Joensuu, Finland}
\email{talponen@iki.fi}

\date{\today}

\begin{abstract}
We study a Bishop-Phelps-Bollob\'as type property for Banach space operators introduced by Dantas (2017). In that paper there is a local and a global version of a natural property which is somewhat similar but simpler compared to the Bishop-Phelps-Bollob\'as type property for operators studied in Acosta et al. (2008). Here we characterize the mentioned local property in the setting with strictly convex domain spaces and compact operators. We show that the local property implies 
that the domain space has strong convexity properties. 
\end{abstract}

\maketitle

\section{Introduction}
The classical Bishop-Phelps theorem states that the norm-attaining functionals are dense in the dual space. Recall that 
$x^* \in \X^*$, an element in the dual space of a Banach space $\X$, is \emph{norm-attaining} if there is $x\in\X$, $\|x\|=1$, such that $\| x^* \| = x^* (x)$. There is a refined version of the classical Bishop-Phelps theorem where one 
has a uniform, even quantitative control, not only on the distance of a given functional to a norm-attaining functional, but also for the witnessing elements $x$ in a suitable sense. This well-known principle is known as the Bishop-Phelps-Bollob\'as theorem appearing in \cite{bollobas}.
Acosta et al. \cite{AAGM} studied the following related property for operators:
\begin{definition}
A pair of Banach spaces $(\X, \Y)$ satisfies the Bishop-Phelps-Bollob\'as property (BPBp) when, given $\vep>0$, there exists $\eta(\vep) >0$ such that whenever $T\in \S_{L(\X,\Y)}$ and $x_0 \in \S_\X$ are such that 
\[\|T(x_0 )\| >1- \eta(\vep)\]
there are $S\in \S_{L(\X,\Y)}$ and $x_1 \in \S_\X$ such that 
\[\|S(x_1 ) \|=1,\quad \|x_1 -x_0 \| <\vep \quad\text{and}\quad \|S-T\|<\vep.\] 
\end{definition}    
\noindent (Some standard notations are recalled in Section \ref{sect: prelim}.)

Thus, in the above definition we allow for small perturbations of the given operator $T$, following the spirit of the Bishop-Phelps-Bollob\'as theorem. It is also natural to ask if the range of 
$T^* \colon \Y^* \to \X^* $ contains many norm-attaining functionals for some fixed operator $T\colon X \to \Y$. 
The genericity of the norm-attaining functionals is investigated in \cite{bourgain}. 
Dantas (2017) \cite{dantas} developes the BPBp analysis to such a direction with $2$ versions of a properties
involving a kind of BPBp for unperturbed operators. There is a weaker local version and a stronger global version of the property
and both properties have convexity implications on the space in question.
The latter property implies that the space $\X$ in question is in fact uniformly convex (cf. \cite{kim}). 
The paper by Dantas is mainly concentrated in the stronger property and here we carry on the analysis regarding the 
weaker Property $1$ of a pair $(\X,\Y)$ of Banach spaces, as it is called in the paper.
Property $1$ states that for any given $T \in \S_{L(\X,\Y)}$ and $\vep>0$ there is $\eta(\vep,T)>0$ such that whenever $x_0 \in \S_{\X}$ 
satisfies 
\[\|T (x_0 )\| > 1-  \eta(\vep,T),\]
there is $x_1 \in \S_\X$ such that 
\[\| T(x_1 )\|=1\quad \text{and}\quad \|x_1 - x_0 \|<\vep.\]

If the above statement holds for all compact operators $T\in \S_{L(\X,\Y)}$ then the pair $(\X, \Y)$ has Property $1$ for compact operators.

Our aim here is to further elucidate Property $1$. Since these properties fundamentally involve convexity, we will study 
them in a strictly convex setting. We will characterize the compact operator version of Property $1$ for strictly convex spaces $\X$. It turns out that that although Property $1$ of $(\X , \Y)$ does not imply the uniform convexity of $\X$, it comes somewhat close to doing that.

\subsection{Preliminaries}\label{sect: prelim}

We denote by $\X$ and $\Y$ real Banach spaces where we exclude the trivial space, consisting of the origin.
Often we impose these conventions implicitly, hence we do not repeat the assumptions on $\X$ and $\Y$ in each result
separately. 

The closed unit ball and the unit sphere of $\X$ are $\B_\X$ and $\S_\X$ respectively.
The space of bounded linear operators $T \colon \X \to \Y$ is denoted by $L(\X, \Y)$. 

See e.g. \cite{FA_book} and \cite{talponen} for general background information on Banach spaces theory and 
for discussion on the extremal structure of unit balls in Banach spaces.
The diameter of a subset $A\subset \X$ is 
\[\diam(A):= \sup_{x,y\in A} \|x-y\|.\]
Given a Banach space $\X$, a slice of its unit ball is defined as follows:
\[S_{f,\alpha} := \{x\in \B_\X \colon 1-\alpha < f(x) \}\] 
where $f\in \S_{\X^*}$ and $0<\alpha<1$.
If the norm of a Banach space is Fr\'echet differentiable away from the origin then the space is said to be 
Fr\'echet smooth. Similar convention holds for Gateaux smoothness. We will apply the related Smulyan
lemma frequently (see \cite{FA_book}).

\section{Results}

\begin{theorem}\label{thm: comp}
Let $\X$ and $\Y$ be Banach spaces and suppose that $\X^*$ is Fr\'echet smooth. Then $(\X,\Y)$ has Property $1$ for compact operators.  
\end{theorem}

The Fr\'echet smoothness of $\X^*$ implies reflexivity of the spaces and it is satisfied for instance if $\X$ is a reflexive LUR space.

The statement of the result does not remain true general operators. For instance, we may choose $\X$ to be a reflexive LUR space 
such that $(\X ,\ell^2)$ fails Property $1$. Such an example is 
\[\X = \ell^1 \oplus_2 \ell^2 \oplus_2 \ell^3 \oplus_2 \ldots \]
with the operator $T \colon \X \to \ell^2$, $T\colon ((x_{k}^{(i)})_k )_i \mapsto (y_j )_j$ given by 
\[T \left[((x_{k}^{(i)})_k )_i \right]= (x_{1}^{(j)})_j .\]

\begin{proof}[Proof of Theorem \ref{thm: comp}]
Let $\X$ and $\Y$ be as in the assumptions and suppose that the statement fails. 
Then there is an operator $T\in \S_{L(\X ,\Y)}$, $\varepsilon>0$ and a sequence of points $(z_k ) \subset \S_\X$ 
such that  $\|T z_k \| \to 1$ and for each $k$ it holds that there does \emph{not} exist $x_k \in \S_\X$ such that 
\begin{equation}\label{eq: Txk}
\|T x_k \|=1 \quad \text{and}\quad \|z_k - x_k \|<\varepsilon .
\end{equation}
According to the compactness of the operator $T$ there is a subsequence $(k_j)\subset \N$ and $y_0 \in \S_\Y$ 
such that $T z_{k_j} \to y_0$ as $j \to \infty$. 

Let $y^* \in \S_{\Y^*}$ be such that $y^* (y_0 )=1$. 
Put $x^* := y^* \circ T \in \S_{\X^*}$. Here 
\begin{equation}\label{eq: xzkj}
x^* (z_{k_j} )\to 1\ \text{as}\ j\to\infty .
\end{equation}
According to the Fr\'echet smoothness of $\X^*$ the space $\X$ is reflexive (see e.g. \cite{talponen}). Thus the functional $x^*$ attains its norm, i.e. there is $x\in \S_\X$ with $x^* (x)=1$. According to \eqref{eq: xzkj}, the Fr\'echet smoothness of $\X^*$, and the Smulyan Lemma (applied on $\X^*$ and $\X^{**}$) it follows that $z_{k_j} \to x$ as $j\to\infty$.
By choosing $x_{k_j}=x$ we arrive at a contradiction (recall \eqref{eq: Txk}) for sufficiently large $j$. Consequently, the statement of the theorem holds.
\end{proof}

\begin{proposition}\label{prop: FS}
Let $\X$ be a strictly convex Banach space. If $(\X,\R)$ has Property $1$, then $\X^*$ is Fr\'echet smooth.
\end{proposition}
\begin{proof}
It is known that if $(\X,\R)$ has Property $1$ then $\X$ is reflexive (see \cite{dantas}). Thus, by the reflexivity and strict convexity of $\X$,
each functional $x^* \in \S_{\X^*}$ attains its norm at exactly one corresponding point $x\in \S_\X$. 
We study an operator $T \colon \X \to \R$ as in the definition of Property $1$. This can be identified with $x^* \in \S_{\X^*}$. 

The condition provided by Property $1$ means that the slices $S_{x^* , \varepsilon}$ shrink uniformly to the corresponding point $x$:
\[\diam (S_{x^* , \eta})\stackrel{\eta\to 0^+}{\longrightarrow} 0,\quad 
\bigcap_{\eta>0} S_{x^* , \eta} =\{x\}.\]
An application of the Smulyan lemma then finnishes the argument.
\end{proof}

It is known that $(\ell^\infty (n) ,\Y)$ has Property $1$ for any Banach space $\Y$ (see Thm. 2.4 in \cite{dantas}). This example shows that if the strict convexity is removed from the assumptions the Fr\'echet smoothness is lost in the conclusion. In fact, there is another route to verifying this claim directly, namely, the G\^{a}teaux smoothness of $\X^*$ already implies the strict convexity 
of $\X$.

\begin{theorem}\label{thm: equiv}
Given spaces $\X$ and $\Y$ the following conditions are equivalent:
\begin{enumerate}
\item $\X$ is strictly convex and $(\X,\Y)$ has Property $1$ for compact operators,
\item $\X^*$ is Fr\'echet smooth.
\end{enumerate}
\end{theorem}
\begin{proof}
The justification of the statement follows from the previous remark together with Theorem \ref{thm: comp} and Proposition \ref{prop: FS}. Indeed, each continuous linear operator $\X \to \R$ is compact and Property $1$ for compact operators of the pair
$(\X,\Y)$ implies that of $(\X,\R)$. 
\end{proof}

The following result suggests that Property $1$ comes in a sense close to uniform convexity of the domain space.
\begin{theorem}
Let $\X$ be a strictly convex Banach space such that $(\X , c_0 )$ has Property $1$. Suppose that a sequence 
$(f_n ) \subset \S_{\X^*}$ satisfies that $f_n \stackrel{\omega^*}{\longrightarrow} 0$ as $n\to\infty$. Then
\[\lim_{\eta\to 0^+} \sup_n \diam(S_{{f_n},\eta})=0.\]
\end{theorem}
\begin{proof}
Let $\X$ and $(f_n )$ be as in the assumptions. Assume to the contrary that
\begin{equation}\label{eq: etasup} 
\lim_{\eta\to 0^+} \sup_n \diam(S_{{f_n},\eta})=d >0 
\end{equation}
where the limit exists, since the diameters shrink as $\eta$ tends to $0$.
According to Theorem \ref{thm: equiv} we have that $\X^*$ is Fr\'echet smooth.
By using the reflexivity of $\X$ and the Smulyan lemma as above we see that
\[\lim_{\eta\to 0^+} \diam(S_{{f_n},\eta})=0 \]
for each $n\in\N$. Together with \eqref{eq: etasup} this yields that there is a subsequence $(n_k) \subset \N$ such that 
\begin{equation}\label{eq: diamd2}
\diam\left(S_{{f_{n_k}},\frac{1}{k}}\right) > \frac{d}{2}
\end{equation}
for each $k\in\N$. Put $d_0 = \min (\frac{d}{2},\frac{1}{4})$.

Note that for each $f_n$ there is a unique $z_n \in \S_\X$ such that $f_n (z_n )=1$. By using the $\omega^*$-convergence of $(f_n )$ we may pass on to a further subsequence $(n_{k_j}) \subset \N$, denoting 
$g_j = f_{n_{k_j}}$ and $y_j = z_{n_{k_j}}$ for $j\in\N$, such that 
\begin{equation}\label{eq: gjyi}
g_j (y_i) < 1 - 2d_0\ \text{for}\ i<j.
\end{equation}

Define $T\colon \X \to c_0$ by $x \mapsto (g_j (x))_j$. Indeed, this definition is proper, since 
$g_j$ $\omega^*$-converges to $0$. Note that $T \in \S_{L(\X, c_0 )}$. 

We claim that $T$ does not satisfy the condition in Property $1$.  
Indeed, fix 
\[0<\eta<1\quad \text{and}\ 0<\varepsilon<\frac{d_0 }{2}. \]
Then we will find $x_0 \in \S_\X$ such that $\|T x_0 \| > 1-\eta$ 
and for each $x_1 \in \S_X$ with $\|T x_1 \|=1$ it holds that $\|x_0 - x_1 \|>\varepsilon$. 
 
Since $g_j (x_1 ) \to 0$ as $j\to\infty$ for any $x_1 \in\X$, it is clear that $\|T x_1 \|=1$ for a given $x_1 \in\S_\X$ if and only if $x_1=y_j$ for some $j\in\N$. Fix $j\in\N$ such that $\frac{1}{j}<\eta$.
Pick $x_0 \in \S_\X$ such that 
\begin{equation}\label{eq: norm}
\|x_0 - y_j \| = \frac{d_0 }{2} >\varepsilon
\end{equation}
and $g_j (x_0 )\geq 1- d_0$. Indeed, this is possible due to \eqref{eq: diamd2}. Then for $ i<j$
\[\|x_0 - y_i \| \geq |g_j (x_0 - y_i )| \geq  \left|(1 - \frac{d_0 }{2})- (1-2d_0 )\right|>d_0 >\varepsilon\]
where we applied \eqref{eq: gjyi}.
The case with $j<i$ is seen as follows:
\begin{multline*}
\|x_0 - y_i \| \geq \|y_j - y_i \| - \|x_0 - y_j \| \geq  |g_i (y_j - y_i )| - \frac{d_0 }{2}\\ 
\geq |(1-2d_0 )-1| - \frac{d_0 }{2}  > \frac{d_0 }{2} >\varepsilon.
\end{multline*}
This contradicts Property $1$ and the proof is complete. 
\end{proof}

We do not known how the statement improves above if we require $(\X ,\ell^\infty)$ to have Property $1$.

\end{document}